\documentclass[a4paper,oneside,reqno]{amsart}

\calclayout

\usepackage{etoolbox}

\newtoggle{default}\toggletrue{default}
\newtoggle{birk}\togglefalse{birk}
\newtoggle{birk_t2}\togglefalse{birk_t2}
\newtoggle{siam}\togglefalse{siam}
\newtoggle{springer}\togglefalse{springer}

\newtoggle{endwithsymbol}\togglefalse{endwithsymbol}

\toggletrue{default}
\toggletrue{endwithsymbol}

\usepackage[utf8]{inputenc}
\usepackage[english]{babel}
\usepackage{microtype} 

\usepackage{amsmath,amssymb,amsthm}
\usepackage{thmtools} 

\usepackage{hyperref} 
\hypersetup{hidelinks}
\usepackage[nameinlink,sort]{cleveref} 
\usepackage{orcidlink}

\usepackage{mathtools} 
\usepackage{upgreek}
\usepackage{accents} 
\usepackage{stmaryrd} 

\usepackage{tikz} 
\usetikzlibrary{cd}
\usetikzlibrary{calc}

\usepackage{dsfont} 
\usepackage{pifont} 

\usepackage{ifthen}
\usepackage{enumitem}
\setlist[itemize]{labelindent=\parindent,leftmargin=*,itemsep=3pt,topsep=5pt}
\setlist[enumerate]{label=\textup{(\roman{*})},itemsep=3pt,labelindent=0pt,leftmargin=*}

\usepackage{nicematrix}

\usepackage{graphicx} 
\usepackage{subcaption}

\usepackage{booktabs} 
\usepackage{esint} 

\usepackage{bbm}

\usepackage{color}

\usepackage{pgfpages} 
\usepackage{pgfplots} 
\pgfplotsset{compat=1.13} 

\definecolor{NathanaelColor}{rgb}{0.2,0.0,0.6}

\newcommand{\R}{\mathbb{R}} 
\newcommand{\C}{\mathbb{C}} 

\newcommand{\Cc}[1][\infty]{\mathrm{C}_{\mathrm{c}}\ifthenelse{\equal{#1}{}}{}{^{#1}}}
\newcommand{\Lp}[2][]{\mathrm{L}_{#2\ifthenelse{\equal{#1}{}}{}{,#1}}} 

\newcommand{\iu}{\mathrm{i}} 
\newcommand{\e}{\mathrm{e}}

\DeclareMathOperator{\ran}{ran}

\DeclareMathOperator{\dom}{dom}
\DeclareMathOperator{\diag}{diag}

\DeclareMathOperator{\curl}{curl}

\DeclarePairedDelimiterX{\dset}[2]{\{}{\}}{#1\,\delimsize\vert\,\mathopen{} #2}
\DeclarePairedDelimiterX{\scprod}[2]{\langle}{\rangle}{#1,#2}

\renewcommand{\Re}{\operatorname{Re}}
\renewcommand{\Im}{\operatorname{Im}}

\theoremstyle{plain}
\newtheorem{theorem}{Theorem}[section]
\newtheorem{lemma}[theorem]{Lemma}
\newtheorem{proposition}[theorem]{Proposition}

\newtheorem{hypothesis}[theorem]{Hypothesis}

\iftoggle{endwithsymbol}{%
    \declaretheorem[style=definition,sibling=theorem,qed=\ding{169}]{definition}
    \declaretheorem[style=definition,sibling=theorem,qed=\ding{169}]{example}
    \declaretheorem[style=definition,sibling=theorem,qed=\ding{169}]{problem}
    \declaretheorem[style=definition,sibling=theorem,qed=\ding{169}]{assumption}
    \declaretheorem[style=definition,numbered=no,qed=\ding{169}]{claim}

    \declaretheorem[style=remark,sibling=theorem,qed=\ding{169}]{remark}
  }%
  {%
    \theoremstyle{definition}

    \theoremstyle{remark}
    \newtheorem{remark}[theorem]{Remark}
  }%

\begin{document}

\title[Exponential Stability for Maxwell]{Exponential Stability for Maxwell-type Systems Revisited} 

%

%
%

\author[M.~Waurick]{Marcus Waurick\,\orcidlink{0000-0003-4498-3574}}
\email{marcus.waurick@math.tu-freiberg.de}

\address{TU Bergakademie Freiberg \\
  Institute of Applied Analysis \\
  Akademiestrasse 6 \\
  D-09596 Freiberg \\
  Germany}

\date{\today}
\dedicatory{}



\ifboolexpr{togl{birk_t2} or togl{birk}}{%
}{%
}%


\begin{abstract} Considering a two-by-two block operator matrix system of Maxwell type, we present an elementary way of deducing exponential stability under minimal smoothness (and boundedness) requirements of the underlying domains when applications are concerned. The approach is based on resolvent estimates using block operator matrices. \end{abstract}

\ifboolexpr{togl{default} or togl{birk_t2} or togl{birk}}{\maketitle}{}%

 \section*{Acknowledgments}
  The present research is based on a review I had to provide for reference \cite{EKL24}. In this paper a similar problem was considered and my report was too harsh in assessing that research. Due to a reasonable editor, the research was still accepted. Nonetheless, it might make sense to provide the elementary functional analytic perspective, I had in mind tackling the problem alluded to in \cite{EKL24}.
  
\section{Introduction}\label{sec:intro}

This note is devoted to the study of exponential stability of systems of the form
\begin{equation}\label{eq:absMax}
\begin{cases}  \big(\partial_t \begin{pmatrix} \alpha & 0 \\ 0 & \beta \end{pmatrix} + \begin{pmatrix} \gamma & 0 \\ 0 & 0 \end{pmatrix}+ \begin{pmatrix} 0 & -C^* \\ C & 0 \end{pmatrix}\big)U = \begin{pmatrix} 0 \\ 0 \end{pmatrix}, & \text{ on }(0,\infty)\times H_0\times H_1 \\
U(0) = \begin{pmatrix} u_0 \\ v_0 \end{pmatrix} \in H_0\times H_1 \end{cases}
\end{equation}
where $H_0,H_1$ are Hilbert spaces, $\alpha=\alpha^*\in L(H_0)$, $\beta=\beta^* \in L(H_1)$, $\gamma\in L(H_0)$ be bounded linear operators. Assume that $2\Re \gamma=\gamma+\gamma^*\geq 2c$ for some $c>0$ in the sense of positive definiteness. Furthermore, $C\colon \dom(C)\subseteq H_0\to H_1$ is closed and densely defined and $C^*$ is the corresponding Hilbert space adjoint. Well-posedness of \cref{eq:absMax} can guaranteed using $C_0$-semigroups or the notion of evolutionary equations. We will present the semi-group perspective below and thus may employ the notion of mild solutions also in the subsequent theorem statement.

In the following we provide an operator-theoretic proof of the following:
\begin{theorem}[{{\cite{EKL24} or \cite{Tro15,DIW24}}}]\label{thm:expstab} Let $(u_0,v_0)\in (H_0,\beta^{-1}\ran(C))$. Then there is $\delta>0$ such that for any mild solution $U \in C[0,\infty;H_0\times H_1)$ of \cref{eq:absMax}, we have
\[
   \|U(t)\|_{H_0\times H_1}\leq \e^{-\delta t}\|(u_0,v_0)\|_{H_0\times H_1}.
\]
\end{theorem}

This theorem in its present form has been provided in \cite{EKL24}. It does however also follow from the perspective provided in \cite[Chapter 11]{STW22} rooted in \cite{Tro13,Tro15} with some minor modifications that can be found in \cite{DIW24}. Since the well-posedness approach chosen in \cite{STW22} uses evolutionary equations and the one in \cite{EKL24} $C_0$-semigroups, we feel the necessity of showing how the tools developed in \cite{STW22} can be applied here. More so, we shall specialise the general results for evolutionary equations to the present case and provide an independent proof of \Cref{thm:expstab}. However, note that the ansatz in \cite{EKL24} is similar as both approaches are based on a change of variables even though \cite{EKL24} stays in time domain and the evolutionary equations approach uses the frequency domain. In passing we note that the perspective offered here can directly be used to generalise to (time-)nonlocal equation with memory terms. We highlight the respective potential of generalisation further down below and will also state the impact of the techniques employed on evolutionary equations in the sense of Picard, see \cite{Pic09} and \cite{STW22}. We emphasise that this small note is providing an elementary functional analytic approach towards exponential stability of the above equation. It is by no means an alternative to results requiring only partial damping, see the classic \cite{Leb96}, where partial damping can imply exponential stability. The present manuscript does stress the applicability of abstract arguments for simple situations though.

The next section, \Cref{sec:nc}, is a prerequisite confirming that we might consider the case $\alpha=1$ and $\beta=1$ without loss of generality. \Cref{sec:wp} provides the well-posedness result for \cref{eq:absMax} employing the Lumer--Phillips theorem. Whilst \Cref{sec:nc} reduces the complexity of the coefficients, \Cref{sec:red} is to reduce the complexity of the operator $C$ (and $C^*$) by using an abstract Helmholtz decomposition. This system is then diagonalised offering one equation without unbounded operators and a 2-by-2 block with off diagonal entries albeit being unbounded yet boundedly invertible. The latter system is then discussed in detail in \Cref{sec:fulldamp_abs}, where the decisive estimate is provided obtained with the help of a change of variables. While the sections before contain standard procedures carried out explictly for the reader's convenience, the contents of \Cref{sec:fulldamp_abs} does contain the actual estimate needed. It is a special case of the rationale provided in \cite{Tro15} (and also comprehensively described in \cite[Chapter 11]{STW22}). In \Cref{sec:proof} we provide a proof of \Cref{thm:expstab} and mention possible generalisations to evolutionary equations. \Cref{sec:fulldamping} offers an example in the context of Maxwell's equations.

\section{Reduction to nice coefficients}\label{sec:nc}

We start off with a reformulation making the theory of $C_0$-semigroups applicable. Thus, in this section, we set the stage to confirm existence and uniqueness of solutions in the next section. The proof of this reformulation is however self-evident. Throughout, let $H_0,H_1$ be Hilbert spaces and $\alpha=\alpha^*, \gamma \in L(H_0)$, $\beta=\beta^*\in L(H_1)$ satisfying $\alpha, \beta \geq c$ and $\Re \gamma\geq 0$. Moreover, let $C\colon \dom(C)\subseteq H_0\to H_1$ be densely defined and closed.

\begin{lemma}\label{lem:ref} (a) Let $U\in C^1[0,\infty;H_0\times H_1)\cap C[0,\infty; \dom(C)\times \dom(C^*))$. If for all $t\in (0,\infty)$
\[
 \big(\begin{pmatrix} \alpha & 0 \\ 0 & \beta \end{pmatrix} U\big)'(t) = \big( \begin{pmatrix} \gamma & 0 \\ 0 & 0 \end{pmatrix}+ \begin{pmatrix} 0 & -C^* \\ C & 0 \end{pmatrix}\big) U(t),
\]
then \begin{multline*}
\tilde{U} \coloneqq \diag(\sqrt{\alpha},\sqrt{\beta})U  \\ \in C^1[0,\infty;H_0\times H_1)\cap C[0,\infty; \dom(\sqrt{\beta}^{-1}C\sqrt{\alpha}^{-1})\times \dom(\sqrt{\alpha}^{-1}C^*\sqrt{\beta}^{-1}))\end{multline*} and 
\[
\tilde{U}'(t) = -\big(\begin{pmatrix} \sqrt{\alpha}^{-1}\gamma\sqrt{\alpha}^{-1} & 0 \\ 0 & 0 \end{pmatrix}+ \begin{pmatrix} 0 & -\sqrt{\alpha}^{-1}C^*\sqrt{\beta}^{-1} \\ \sqrt{\beta}^{-1}C\sqrt{\alpha}^{-1} & 0 \end{pmatrix}\big)\tilde{U}(t)\quad(t\in (0,\infty)).
\]
(b) Let $\tilde{U}  \in C^1[0,\infty;H_0\times H_1)\cap C[0,\infty; \dom(\sqrt{\beta}^{-1}C\sqrt{\alpha}^{-1})\times \dom(\sqrt{\alpha}^{-1}C^*\sqrt{\beta}^{-1}))$. If for all $t\in (0,\infty)$
\[
\tilde{U}'(t) = -\big(\begin{pmatrix} \sqrt{\alpha}^{-1}\gamma\sqrt{\alpha}^{-1} & 0 \\ 0 & 0 \end{pmatrix}+ \begin{pmatrix} 0 & -\sqrt{\alpha}^{-1}C^*\sqrt{\beta}^{-1} \\ \sqrt{\beta}^{-1}C\sqrt{\alpha}^{-1} & 0 \end{pmatrix}\big)\tilde{U}(t),
\]
then $U\coloneqq \diag(\sqrt{\alpha}^{-1},\sqrt{\beta}^{-1})\tilde{U}\in C^1[0,\infty;H_0\times H_1)\cap C[0,\infty; \dom(C)\times \dom(C^*))$
and
\[
 \big(\begin{pmatrix} \alpha & 0 \\ 0 & \beta \end{pmatrix} U\big)'(t) = \big( \begin{pmatrix} \gamma & 0 \\ 0 & 0 \end{pmatrix}+ \begin{pmatrix} 0 & -C^* \\ C & 0 \end{pmatrix}\big) U(t).
\]
\end{lemma}

\begin{remark}\label{rem:easy} In the situation of the previous lemma, the following statements hold, which we present here without a proof. Note that by changing the scalar-product appropriately $\sqrt{\alpha}$ and $\sqrt{\beta}$ (and their respective inverses) can be established as unitary operators.

\begin{enumerate}
\item[(a)] We have 
\[
   \sqrt{\alpha}^{-1}C^*\sqrt{\beta}^{-1} =    \sqrt{\beta}^{-1}C\sqrt{\alpha}^{-1}.
\]
\item[(b)] Let $\gamma\in L(H_0)$. Then $\Re \gamma\geq 0$ if and only if $\Re  \sqrt{\alpha}^{-1}\gamma\sqrt{\alpha}^{-1}\geq 0$. Moreover,
\[
  \exists c>0\colon \Re \gamma\geq c \iff \exists c>0\colon \Re \sqrt{\alpha}^{-1}\gamma\sqrt{\alpha}^{-1}\geq c.
\]
\item[(c)] $\ran(C)\subseteq H_1$ is closed if and only if $\ran( \sqrt{\beta}^{-1}C\sqrt{\alpha}^{-1})\subseteq H_1$ is closed.
\item[(d)] Let $v_0\in H_1$. Then $v_0\in \beta^{-1}\ran(C)$ if and only if $\tilde{v}_0\coloneqq \sqrt{\beta}v_0 \in \sqrt{\beta}^{-1}\ran(C)=\ran(\sqrt{\beta}^{-1}C\sqrt{\alpha}^{-1})$.
\item[(e)] Stability estimates of the form in \Cref{thm:expstab} for $U$ yield the corresponding one for $\tilde{U}$ and vice versa. \qedhere
\end{enumerate}
\end{remark}

As a consequence, as long as we are staying in the abstract Hilbert space setting, we may assume without loss of generality, that $\alpha=1$ and $\beta=1$. 

Next, we turn to existence and uniqueness of the above solutions.

\section{Well-posedness and the path to exponential stability}\label{sec:wp}

Throughout this section, let $H_0$, $H_1$ be Hilbert spaces, $\gamma\in L(H_0)$ be bounded with $\Re \gamma\geq 0$ and $C\colon \dom(C)\subseteq H_0\to H_1$ be densely defined and closed.

\begin{remark}\label{rem:mdiss} Let $H$ be a Hilbert space and $\eta\in L(H)$ such that $\Re \eta\geq 0$ and $A\colon \dom(A)\subseteq H\to H$ be skew-selfadjoint. Then $B\coloneqq -(\eta+A)$ is m-dissipative. Indeed, the dissipativity is easy to verify. Next, $B^*=-\eta^*+A$ is also m-dissipative. As $B$ is closed, it follows that $B$ is m-dissipative.
\end{remark}
The latter remark thus provides an elementary example class of operators generating a $C_0$-semigroup of contractions by the Lumer--Phillips theorem.
\begin{proposition}\label{prop:genthm} The operator
\[
   B\coloneqq -\begin{pmatrix} \gamma & 0 \\ 0 & 0\end{pmatrix} - \begin{pmatrix} 0 & -C \\ C^* & 0\end{pmatrix} 
\]
is m-dissipative and, hence, generates a $C_0$-semigroup of contractions.
\end{proposition}
\begin{proof}
It is elementary, see, e.g., \cite[Proposition 6.2.3]{STW22}, to show that $A= \begin{pmatrix} 0 & -C \\ C^* & 0\end{pmatrix} $ is skew-selfadjoint. Since $\eta=\begin{pmatrix} \gamma & 0 \\ 0 & 0\end{pmatrix}$ satisfies $\Re \eta\geq 0$, the assertion follows from \Cref{rem:mdiss}.
\end{proof}

For the proof of exponential stability, we use the seminal Gearhart--Pr\"uss theorem, which we formulate here for the special case of m-dissipative operators.

\begin{theorem}[\cite{Pr84}]\label{thm:GP} Let $B$ be m-dissipative on a Hilbert space $H$. Then $B$ generates an exponentially stable semi-group if and only if
\[
\iu \R \subseteq \rho(B) \text{ and }\sup_{\lambda\in \R} \|(\iu\lambda-B)^{-1}\|<\infty.\]
\end{theorem}

Thus, we aim to find estimates so that we may apply the Gearhart--Pr\"uss Theorem. For this, we reformulate the resolvent.

\section{A reduction procedure}\label{sec:red}

The following section addresses only a prototype situation, which leads us to the general result anyway. We adopt the notation of the previous section.
\begin{hypothesis}\label{hyp:fne0p} Assume that $\alpha=1$ and $\beta=1$ and $\gamma\in L(H_0)$ such that $\Re\gamma\geq c$ for some $c>0$.\end{hypothesis}
\begin{hypothesis}\label{hyp:clr} $\ran(C)\subseteq H_1$ is closed.
\end{hypothesis}
\begin{remark}\begin{enumerate}
\item[(a)] The closed range theorem, see, e.g., \cite[Theorem IV.1.2]{Gol06}, asserts that $\ran(C)\subseteq H_1$ is closed if and only if $\ran(C^*)\subseteq H_0$ is. As a consequence, \Cref{hyp:clr} implies that both $\ran(C)\subseteq H_1$ and $\ran(C^*)\subseteq H_0$ are closed.
\item[(b)] We have the abstract Helmholtz-type decompositions
\[
   H_1 = \ran(C)\oplus \ker(C^*)\text{ and }H_0 = \ran(C^*)\oplus \ker(C).
\]
\item[(c)] Referring to the FA-toolbox, e.g., in \cite{PZ20}, a sufficient condition fo $C$ to have closed range is that $\dom(C)\cap \ker(C)^\bot \hookrightarrow H_0$ is compact.
\item[(d)] To sketch the idea of (c) note that an operator to have closed range is equivalent to the validity of a \textbf{closed range inequality for $C$}, that is,
\[
\exists c>0  \forall x\in \dom(C)\cap \ker(C)^\bot \colon \|x\|\leq c\|Cx\|,
\]
see again \cite{Gol06} or, see also \cite{PW26} for more details and suitable references in this context.\qedhere
\end{enumerate}
\end{remark}

We introduce $\iota_1 \colon \ran(C)\hookrightarrow H_1$, the continuous embedding. If \Cref{hyp:clr} holds, then $\iota_1^*\colon H_1\to \ran(C)$ is the (surjective) orthogonal projection. We denote by $\kappa_1\colon \ker(C^*)\hookrightarrow H_1$ the canonical embedding. Similarly, we let $\iota_0\colon \ran(C^*)\hookrightarrow H_0$ and $\kappa_0\colon \ker(C)\hookrightarrow H_0$.

Moreover, we define for $z\in \C$
\[
   B_z \coloneqq z + \begin{pmatrix} \gamma & 0 \\ 0 & 0 \end{pmatrix} +\begin{pmatrix} 0 & C^* \\ -C & 0 \end{pmatrix}.
\]
\begin{theorem}\label{thm:soldecom} Assume \Cref{hyp:fne0p} and \Cref{hyp:clr} hold. Let $f\in H_0, g\in \ran(C)$, $z\in \C$, $(u,v)\in \dom(C)\times \dom(C^*)\cap \ran(C)$. Then the following conditions are equivalent:
\begin{enumerate}
\item[(i)] $B_z(u,v)=(f,g)$.
\item[(ii)] $u=\iota_0\iota^*_0 u+  \kappa_0\kappa^*_0 u$ and $v=\iota_1\iota^*_1 v+  \kappa_1\kappa^*_1 v$ satisfy
\[
\Big(z\begin{pmatrix}1 & 0&0\\ 0 & 1& 0 \\ 0 & 0 & 1 \end{pmatrix}+
\begin{pmatrix} \iota_0^*\gamma \iota_0 & 0&\iota_0^*\gamma \kappa_0\\ 0 & 0& 0 \\ \kappa_0^*\gamma \iota_0 & 0 & \kappa_0^*\gamma \kappa_0 \end{pmatrix} + \begin{pmatrix}0 & -\iota_0^*C^*\iota_1 &0\\ \iota_1^*C \iota_0 & 0& 0 \\ 0 & 0 & 0 \end{pmatrix}\Big)\begin{pmatrix} \iota_0^*u \\ \iota_1^*v \\ \kappa_0^*u\end{pmatrix} = \begin{pmatrix} \iota_0^*f \\ \iota_1^*g \\ \kappa_0^*f\end{pmatrix}.
\]
\end{enumerate} 
\end{theorem}
\begin{proof}
Assume (i). Then with $F=(f,g)$ and $U=(u,v)$,
\[
 \big(z \begin{pmatrix} 1 & 0 \\ 0 & 1 \end{pmatrix}  + \begin{pmatrix}  \gamma & 0 \\ 0 & 0 \end{pmatrix}\\ + \begin{pmatrix} 0 & -C^* \\ C & 0 \end{pmatrix}\big)U=F.
\]
Then reading the equation line by line yields
\begin{equation}\label{eq:linebyline}
    z u + \gamma u -C^*v = f\text{ and } z v + Cu =g.
\end{equation}
Since $g\in \ran(C)$ and $v\in \ran(C)$, we deduce $v=\iota_1\iota_1^*v$ and $g=\iota_1\iota_1^*g$. Thus, we may write the second equation in \cref{eq:linebyline} as follows
\[
z\iota_1^*v  = -\iota_1^*C \iota_0 \iota_0^*u + \iota_1^*g.
\]Moreover, we write $u=\iota_0\iota^*_0 u+  \kappa_0\kappa^*_0 u$ and obtain
\begin{align*}
f & = 
z \iota_0\iota^*_0 u+ z\kappa_0\kappa^*_0 u + \gamma (\iota_0\iota^*_0 u+  \kappa_0\kappa^*_0 u) -C^*v \\
  & = 
  z \iota_0\iota^*_0 u+  z\kappa_0\kappa^*_0 u + (\iota_0\iota^*_0 + \kappa_0\kappa^*_0)\gamma (\iota_0\iota^*_0 u+  \kappa_0\kappa^*_0 u) -C^*\iota_1\iota_1^*v.
\end{align*}
Thus, we obtain
\[
  \iota_0^*f =  z \iota^*_0 u + \iota^*_0\gamma (\iota_0\iota^*_0 u+  \kappa_0\kappa^*_0 u) -\iota_0^*C^*\iota_1\iota_1^*v.
\]
and
\[
\kappa_0^*f =  z\kappa^*_0 u + \kappa^*_0\gamma (\iota_0\iota^*_0 u+  \kappa_0\kappa^*_0 u). 
\]Altogether, we obtain the following matrix form
\[
\Big(z \begin{pmatrix}1 & 0&0\\ 0 & 1& 0 \\ 0 & 0 & 1 \end{pmatrix}+
\begin{pmatrix} \iota_0^*\gamma \iota_0 & 0&\iota_0^*\gamma \kappa_0\\ 0 & 0& 0 \\ \kappa_0^*\gamma \iota_0 & 0 & \kappa_0^*\gamma \kappa_0 \end{pmatrix} + \begin{pmatrix}0 & -\iota_0^*C^*\iota_1 &0\\ \iota_1^*C \iota_0 & 0& 0 \\ 0 & 0 & 0 \end{pmatrix}\Big)\begin{pmatrix} \iota_0^*u \\ \iota_1^*v \\ \kappa_0^*u\end{pmatrix} = \begin{pmatrix} \iota_0^*f \\ \iota_1^*g \\ \kappa_0^*f\end{pmatrix}.\qedhere
\]
\end{proof}

The upshot of this reformulation is that modulo an additional equality we may assume that $C$ and $C^*$ are \textbf{both one-to-one and onto}. The block structure of the equation at hand allows to simplify the equation even further:

\begin{theorem}\label{thm:diag} Let $z\in \C_{\Re>-c}$. Assume \Cref{hyp:fne0p} and \Cref{hyp:clr} hold. Let $f\in H_0, g\in \ran(C)$, $z\in \C$, $(u,v)\in \dom(C)\times \dom(C^*)\cap \ran(C)$. Then the following conditions are equivalent:
\begin{enumerate}
\item[(i)] $B_z(u,v)=(f,g)$.
\item[(ii)] With $U \coloneqq T_2(z)^{-1}\begin{pmatrix} \iota_0^*u \\ \iota_1^*v \\ \kappa_0^*u\end{pmatrix} $ and $F\coloneqq T_1(z) \begin{pmatrix} \iota_0^*f \\ \iota_1^*g \\ \kappa_0^*f\end{pmatrix}$ we have
\begin{multline*}
\Big(z \begin{pmatrix}1 & 0\\ 0 & 1 \end{pmatrix}+
\begin{pmatrix} \iota_0^*\gamma \iota_0-\iota_0^*\gamma\kappa_0 (z+\kappa_0^*\gamma\kappa_0)^{-1} \kappa_0^*\gamma \iota_0 & 0\\ 0 & 0 \end{pmatrix} \\ +\begin{pmatrix}0 & -\iota_0^*C^*\iota_1\\ \iota_1^*C \iota_0 & 0 \end{pmatrix}\Big)\begin{pmatrix} U_1 \\ U_2 \end{pmatrix}=\begin{pmatrix} F_1 \\ F_2 \end{pmatrix}.
\end{multline*}
and 
\[
(z+\kappa_0^*\gamma\kappa_0)U_3 = F_3
\]
\begin{multline*}
  T_1(z)\coloneqq
\begin{pmatrix}1 & 0&-\iota_0^*\gamma\kappa_0 (z+\kappa_0^*\gamma\kappa_0)^{-1}\\ 0 & 1& 0 \\ 0 & 0 & 1 \end{pmatrix}   \text{ and }\\ T_2(z)\coloneqq  \begin{pmatrix}1 & 0&0 \\ 0 & 1& 0 \\ -(z+\kappa_0^*\gamma\kappa_0)^{-1}\kappa_0^*\gamma\iota_0  & 0 & 1 \end{pmatrix}.
\end{multline*}
\end{enumerate}
\end{theorem}
\begin{proof}
With \Cref{thm:soldecom} we may, equivalently, rewrite (i) as \[
{z \begin{pmatrix}1 & 0&0\\ 0 & 1& 0 \\ 0 & 0 & 1 \end{pmatrix}+
\begin{pmatrix} \iota_0^*\gamma \iota_0 & 0&\iota_0^*\gamma \kappa_0\\ 0 & 0& 0 \\ \kappa_0^*\gamma \iota_0 & 0 & \kappa_0^*\gamma \kappa_0 \end{pmatrix} + \begin{pmatrix}0 & -\iota_0^*C^*\iota_1 &0\\ \iota_1^*C \iota_0 & 0& 0 \\ 0 & 0 & 0 \end{pmatrix}\Big)}\begin{pmatrix} \iota_0^*u \\ \iota_1^*v \\ \kappa_0^*u\end{pmatrix} = \begin{pmatrix} \iota_0^*f \\ \iota_1^*g \\ \kappa_0^*f\end{pmatrix}.
\]
As $\Re\gamma\geq c$ and $\Re z>-c$ we get that $z+\kappa_0^*\gamma\kappa_0$ is continuously invertible (note that, trivially, $\Re \kappa_0^*\gamma\kappa_0\geq c\kappa_0^*\kappa_0$). Next, we introduce $U \coloneqq T_2(z)^{-1}\begin{pmatrix} \iota_0^*u \\ \iota_1^*v \\ \kappa_0^*u\end{pmatrix} $ and $F\coloneqq T_1(z) \begin{pmatrix} \iota_0^*f \\ \iota_1^*g \\ \kappa_0^*f\end{pmatrix}$ and obtain, equivalently,
\begin{multline*}
\Big(z \begin{pmatrix}1 & 0&0\\ 0 & 1& 0 \\ 0 & 0 & 1 \end{pmatrix}+
\begin{pmatrix} \iota_0^*\gamma \iota_0-\iota_0^*\gamma\kappa_0 (z+\kappa_0^*\gamma\kappa_0)^{-1} \kappa_0^*\gamma \iota_0 & 0&0\\ 0 & 0& 0 \\ 0 & 0 & \kappa_0^*\gamma \kappa_0 \end{pmatrix} \\ +\begin{pmatrix}0 & -\iota_0^*C^*\iota_1 &0\\ \iota_1^*C \iota_0 & 0& 0 \\ 0 & 0 & 0 \end{pmatrix}\Big)U=F.
\end{multline*}
This system of equation decouples into
\begin{multline*}
\Big(z \begin{pmatrix}1 & 0\\ 0 & 1 \end{pmatrix}+
\begin{pmatrix} \iota_0^*\gamma \iota_0-\iota_0^*\gamma\kappa_0 (z+\kappa_0^*\gamma\kappa_0)^{-1} \kappa_0^*\gamma \iota_0 & 0\\ 0 & 0 \end{pmatrix} \\ +\begin{pmatrix}0 & -\iota_0^*C^*\iota_1\\ \iota_1^*C \iota_0 & 0 \end{pmatrix}\Big)\begin{pmatrix} U_1 \\ U_2 \end{pmatrix}=\begin{pmatrix} F_1 \\ 0 \end{pmatrix}.
\end{multline*}
and 
\[
(z+\kappa_0^*\gamma\kappa_0)U_3 = F_3.\qedhere
\]
\end{proof}
\begin{remark}\label{rem:rezc}
\begin{enumerate}
\item[(a)] We have already argued that $(z+\kappa_0^*\gamma\kappa_0)$ is continuously invertible. It is easy to see that $\| (z+\kappa_0^*\gamma\kappa_0)^{-1}\|\leq \frac1{\Re z + c}$.
\item[(b)]
In order estimate $\|(B_z)^{-1}\|$, it thus suffices to estimate 
\[
\Big\|\Big(z \begin{pmatrix}1 & 0\\ 0 & 1 \end{pmatrix}+
\begin{pmatrix} \iota_0^*\gamma \iota_0-\iota_0^*\gamma\kappa_0 (z+\kappa_0^*\gamma\kappa_0)^{-1} \kappa_0^*\gamma \iota_0 & 0\\ 0 & 0 \end{pmatrix} \\ +\begin{pmatrix}0 & -\iota_0^*C^*\iota_1\\ \iota_1^*C \iota_0 & 0 \end{pmatrix}\Big)^{-1}\Big\|
\]
and
\[
\|(z+\kappa_0^*\gamma\kappa_0)^{-1}\|.
\]
Indeed, this follows from \Cref{thm:diag} and the fact that
\[
    \|T_1(z)\|,     \|T_1(z)^{-1}\|,     \|T_2(z)\|,     \|T_2(z)^{-1}\| \leq 1+\frac1{\Re z + c}\|\gamma\|.
\]
So, as long as $\Re z$ is uniformly bounded away from $-c$, the statement on the estimation of $\|(B_z)^{-1}\|$ follows.\qedhere
\end{enumerate}
 \end{remark}
By the previous remark, it suffices to treat the first operator in the reformulation (ii) of the previous theorem. For this, the following elementary observation from \cite{DIW24} will be useful:
\begin{lemma}[{{\cite[Lemma 3.9]{DIW24}}}]\label{lem:pddtw} Let $H=H_0\times H_1$ and $\eta\in L(H)$ be such that $\Re \eta\geq c$ for some $c>0$. Then $\Re \eta_{11}\geq c$ and
\[
  \Re  (\eta_{00}-\eta_{01}\eta_{11}^{-1}\eta_{10})\geq c.
\]
\end{lemma}
As a consequence, we obtain the following result, which is a mere reformulation of \Cref{lem:pddtw}.
\begin{lemma}\label{lem:topleft} Let $\gamma\in L(H_0)$ with $\Re \gamma\geq c>0$, $z\in \C$ such that $\Re z>-c$. Then
\[
   \Re \iota_0^*\gamma \iota_0-\iota_0^*\gamma\kappa_0 (z+\kappa_0^*\gamma\kappa_0)^{-1} \kappa_0^*\gamma \iota_0 \geq \min\{\Re z + c, c\}.
\]
\end{lemma}
\begin{proof}
The claim follows by applying \Cref{lem:pddtw} to $\eta = \begin{pmatrix}\iota_0^*\gamma \iota_0 & \iota_0^*\gamma\kappa_0 \\ \kappa_0^*\gamma \iota_0 & z+ \kappa_0^*\gamma\kappa_0\end{pmatrix}$.
\end{proof}
In the next section, we are concerned with the first equation in the decomposition provided in \Cref{thm:diag}.

\section{The special case of one-to-one and onto $C$}
\label{sec:fulldamp_abs}

Let $\gamma\in L(H_0)$ with $\Re \gamma\geq c$. Throughout this section, we let $C\colon \dom(C)\subseteq H_0\to H_1$ be densely defined, closed, one-to-one and onto. 

Let us consider 
\[
B \coloneqq  -  \begin{pmatrix} \gamma & 0 \\ 0 & 0 \end{pmatrix} -   \begin{pmatrix} 0 & -C^* \\ C & 0 \end{pmatrix}.
\]
\begin{theorem}\label{thm:resB} There exists $\delta>0$ such that
\[
   \C_{\Re>-\delta}\subseteq \rho(B)\text{ and } \sup_{z\in \C_{\Re>-\delta}}\|(z-B)^{-1}\|<\infty.
\]
\end{theorem}
The idea of proof is a change of variable ansatz, which, in a more general situation is treated in \cite{Tro15}, see also \cite{STW22}.

\begin{lemma}\label{lem:zne0} Let $z\in \C\setminus \{0\}$, $\delta>0$, $\delta\neq- z$, and $U\coloneqq (u,v)\in \dom(C)\times \dom(C^*)$. Then for $F\coloneqq (f,g)\in H_0\times H_1$ the following conditions are equivalent:
\begin{enumerate}
  \item[(i)] $(z-B)U=F$
  \item[(ii)] With $U_\delta = (u_\delta,v)$ where $u_\delta = u+ \frac{\delta}{z}u$ and $F_\delta = (f+(\gamma-\delta)\frac{\delta}{z}C^{-1}g, (1+\frac{\delta}{z})g)$ we have \[
  \Big(z \begin{pmatrix} 1 & 0 \\ 0 & 1 \end{pmatrix} + \begin{pmatrix} (\gamma-\delta) & (\gamma-\delta){\delta}C^{-1}   \\ 0 & \delta \end{pmatrix} +  \begin{pmatrix} 0 & -C^* \\ C & 0 \end{pmatrix}\Big) U_\delta= F_\delta.
  \]
\end{enumerate}
\end{lemma}
\begin{proof}
Note that (i) holds if and only if
\[
    \Big(z \begin{pmatrix} 1 & 0 \\ 0 & 1 \end{pmatrix} + \begin{pmatrix} \gamma & 0 \\ 0 & 0 \end{pmatrix} +  \begin{pmatrix} 0 & -C^* \\ C & 0 \end{pmatrix}\Big) U= F
\]
Multiplying on the left with the topological isomorphism 
\[
   \begin{pmatrix} 1 &  (\gamma-\delta)\frac{\delta}{z}C^{-1} \\ 0 & 1-\frac{\delta}{z}\end{pmatrix}
\]we get, equivalently,
\begin{align*}
F_\delta & = \begin{pmatrix} 1 &  (\gamma-\delta)\frac{\delta}{z}C^{-1} \\ 0 & 1+\frac{\delta}{z}\end{pmatrix} F \\
   &   =    \begin{pmatrix} 1 &  (\gamma-\delta)\frac{\delta}{z}C^{-1} \\ 0 & 1+\frac{\delta}{z}\end{pmatrix}  \Big(z \begin{pmatrix} 1 & 0 \\ 0 & 1 \end{pmatrix} + \begin{pmatrix} \gamma & 0 \\ 0 & 0 \end{pmatrix} +  \begin{pmatrix} 0 & -C^* \\ C & 0 \end{pmatrix}\Big) U \\
   & =    \Big(z \begin{pmatrix} 1 & 0 \\ 0 & 1 \end{pmatrix} + \begin{pmatrix} \gamma+(\gamma-\delta)\frac{\delta}{z} & (\gamma-\delta){\delta}C^{-1}   \\ 0 & \delta \end{pmatrix} +  \begin{pmatrix} 0 & -C^* \\ C(1+\frac{\delta}{z}) & 0 \end{pmatrix}\Big) U\\
   & =  \Big(z \begin{pmatrix} (1+\frac{\delta}{z}) & 0 \\ 0 & 1 \end{pmatrix} + \begin{pmatrix} (\gamma-\delta)(1+\frac{\delta}{z}) & (\gamma-\delta){\delta}C^{-1}   \\ 0 & \delta \end{pmatrix} +  \begin{pmatrix} 0 & -C^* \\ C(1+\frac{\delta}{z}) & 0 \end{pmatrix}\Big) U \\
   & =  \Big(z \begin{pmatrix} 1 & 0 \\ 0 & 1 \end{pmatrix} + \begin{pmatrix} (\gamma-\delta) & (\gamma-\delta){\delta}C^{-1}   \\ 0 & \delta \end{pmatrix} +  \begin{pmatrix} 0 & -C^* \\ C & 0 \end{pmatrix}\Big) U_\delta.\qedhere
\end{align*}
\end{proof}
\begin{lemma}\label{lem:realpartest} Let $z\in \C\setminus \{0\}$, $\delta>0$. Then
\begin{multline*}
   \Re  \Big(z \begin{pmatrix} 1 & 0 \\ 0 & 1 \end{pmatrix} + \begin{pmatrix} (\gamma-\delta) & (\gamma-\delta){\delta}C^{-1} \\ 0 & \delta\end{pmatrix}\Big)\\ \geq \min\{ \Re z +  (c-\delta(1 + \frac{1}{2}((\|\gamma\|+\delta)\|C^{-1}\|)^2),\Re z+ \frac{\delta}{2}\}.
\end{multline*}
\end{lemma}
\begin{proof}
 Let $(u,v) \in H_0\times H_1$. Then
 \[
  \Re\langle z \begin{pmatrix} u \\ v\end{pmatrix}, \begin{pmatrix} u \\ v\end{pmatrix}\rangle = \Re z (\|u\|^2 + \|v\\^2).
 \]
 Moreover, using $\Re \gamma\geq c$ and Young's inequality for $\varepsilon=p\delta>0$ for some $p>0$, we get
 \begin{align*}
 &\Re  \langle \begin{pmatrix} (\gamma-\delta) & (\gamma-\delta){\delta}C^{-1} \\ 0 & \delta \end{pmatrix}\begin{pmatrix} u \\ v\end{pmatrix}, \begin{pmatrix} u \\ v\end{pmatrix}\rangle   \\ &\geq (c-\delta)\|u\|^2 - (\|\gamma\|+\delta)\delta\|C^{-1}\| \|u\|\|v\| + \delta\|v\|^2 \\
& \geq (c-\delta - \frac{1}{2\varepsilon}((\|\gamma\|+\delta)\delta\|C^{-1}\|)^2)\|u\|^2 + (\delta-\frac{1}{2}\varepsilon)\|v\|^2 \\
& = (c-\delta(1 + \frac{1}{2p}((\|\gamma\|+\delta)\|C^{-1}\|)^2)\|u\|^2 + \delta(1-\frac{1}{2}p)\|v\|^2.
 \end{align*}
 Thus, the assertion follows by setting $p=1$.
\end{proof}
\begin{remark}
From the proof we see that there is room for optimising the lower bound, by optimising in $p$.
\end{remark}

For the case of $z=0$, we require a different strategy than for $z\neq 0$:

\begin{lemma}\label{lem:block} Let $H_0,H_1$ be Hilbert spaces.

(a) Let $A\in L(H_0)$ and let $B\colon \dom(B)\subseteq H_1 \to H_0$ as well as $C\colon \dom(C)\subseteq H_0\to H_1$ be both continuously invertible. Then $ \begin{pmatrix}
    A & B \\ C & 0
   \end{pmatrix} $
on its natural domain $\dom(C)\times \dom(B)$ as an operator in $H_0\times H_1$ is continuously invertible and we have
\[
   \begin{pmatrix}
    A & B \\ C & 0
   \end{pmatrix} ^{-1} = 
   \begin{pmatrix}
    0 & C^{-1} \\ B^{-1} & -B^{-1}AC^{-1}
   \end{pmatrix} 
\] 
\end{lemma}
\begin{proof}
The formula is immediately confirmed by direct computation.
\end{proof}

\begin{proof}[Proof of \Cref{thm:resB}] First of all, note that $B$ is invertible by \Cref{lem:block}. Next, as the resolvent set is open, we find $\varepsilon>0$ such that $B(0,\varepsilon)\subseteq \rho(B)$. Next, let $0<\delta<\varepsilon$ and, by possibly further decreasing $\delta$ we may assume that $ c-\delta(1 + \frac{1}{2}((\|\gamma\|+\delta)\|C^{-1}\|)^2)\coloneqq \tilde{c}>0$. Define $d\coloneqq \frac{1}{2}\min\{\tilde{c},\delta/2\}$. For $z\in \C_{\Re>-d}$, we let $(u,v)\in \dom(B)$ and $F\in H_0\times H_1$ be such that 
\[
   (z-B)U=F.
\]
Then, by \Cref{lem:zne0}, we have
\[
  \Big(z \begin{pmatrix} 1 & 0 \\ 0 & 1 \end{pmatrix} + \begin{pmatrix} (\gamma-\delta) & (\gamma-\delta){\delta}C^{-1}   \\ 0 & \delta \end{pmatrix} +  \begin{pmatrix} 0 & -C^* \\ C & 0 \end{pmatrix}\Big) U_\delta= F_\delta,
  \]
  where $U_\delta$ and $F_\delta$ has been introduced in  \Cref{lem:zne0}. Next, since
  \[
      \iu \Im z + \begin{pmatrix} 0 & -C^* \\ C & 0 \end{pmatrix}
  \]
  is skew-selfadjoint, we deduce that 
  \[
     B_\delta \coloneqq   \Big(z \begin{pmatrix} 1 & 0 \\ 0 & 1 \end{pmatrix} + \begin{pmatrix} (\gamma-\delta) & (\gamma-\delta){\delta}C^{-1}   \\ 0 & \delta \end{pmatrix} +  \begin{pmatrix} 0 & -C^* \\ C & 0 \end{pmatrix}\Big)
  \]
  is continuously invertible as 
  \[
    \Re B_\delta\geq \Re z + \min\{\tilde{c},\frac{\delta}{2}\}\geq d>0;
  \]
  the inverse satisfies
  \[
      \|B_\delta^{-1}\|\leq \frac{1}{d}.
  \]
  Hence,
  \begin{align*}
    \|U\| &\leq \max\{|(1+\delta/z)|^{-1},1\}\|U_\delta\| \\
    & \leq \frac{1}{d}\max\{|(1+\delta/z)|^{-1},1\} \|F_\delta\| \\
    & \leq \frac{1}{d}\max\{|(1+\delta/z)|^{-1},1\} \big((1+\|\gamma\|+\delta)\frac{\delta}{|z|}\|C^{-1}\|+(1+\frac{\delta}{|z|})\big)\|F\| \\
    & \leq \frac{1}{d}2 \big((1+\|\gamma\|+\delta)\|C^{-1}\|+2\big)\|F\|,
  \end{align*}
  providing the desired continuity estimate for the resolvent.
\end{proof}

\section{Proof of the main theorem}\label{sec:proof}

We shall prove the main theorem in the following form. We present the corresponding statement for classical solutions. The case for mild solutions follows from a (standard) continuity argument.

\begin{theorem}\label{thm:mt} Let $\alpha,\beta\in L(H_0)$ be selfadjoint and $\gamma\in L(H_0)$ satisfying the positive definiteness conditions:
\[
   \alpha\geq c, \quad \beta\geq c, \quad \Re \gamma \geq c.
\]
Let $C\colon \dom(C)\subseteq H_0\to H_1$ be densely defined, closed with closed range. Then there exists $\delta>0$ such that for $u_0 \in \dom(C)$ and $v_0 \in \dom(C^*)\cap \beta^{-1}\ran(C)$ and any $U\in C^1[0,\infty; H_0\times H_1)\cap C[0,\infty; \dom(C)\times \dom(C^*)$ satisfying $U(0)=(u_0,v_0)$ and
\[
    \big(\begin{pmatrix} \alpha & 0 \\ 0 & \beta \end{pmatrix} U\big)'(t) = \big( \begin{pmatrix} \gamma & 0 \\ 0 & 0 \end{pmatrix}+ \begin{pmatrix} 0 & -C^* \\ C & 0 \end{pmatrix}\big) U(t)\quad (t\in (0,\infty))
\]
we have
\[
   \|U(t)\|_{H_0\times H_1}\leq \e^{-\delta t}\|(u_0,v_0)\|.
\]
\end{theorem}
\begin{proof}[Proof of \Cref{thm:mt} and \Cref{thm:expstab}] Let $(u_0,v_0)$ and $U$ as in the theorem statement. Then, by \Cref{lem:ref}, $\tilde{U}\coloneqq \diag(\sqrt{\alpha},\sqrt{\beta})U$ satisfies 
\[
\tilde{U}'(t) = -\big(\begin{pmatrix} \sqrt{\alpha}^{-1}\gamma\sqrt{\alpha}^{-1} & 0 \\ 0 & 0 \end{pmatrix}+ \begin{pmatrix} 0 & -\sqrt{\alpha}^{-1}C^*\sqrt{\beta}^{-1} \\ \sqrt{\beta}^{-1}C\sqrt{\alpha}^{-1} & 0 \end{pmatrix}\big)\tilde{U}(t).
\]
and $\tilde{U}(0)= \diag(\sqrt{\alpha},\sqrt{\beta})(u_0,v_0)$. Since 
\[
B\coloneq -\big(\begin{pmatrix} \sqrt{\alpha}^{-1}\gamma\sqrt{\alpha}^{-1} & 0 \\ 0 & 0 \end{pmatrix}+ \begin{pmatrix} 0 & -\sqrt{\alpha}^{-1}C^*\sqrt{\beta}^{-1} \\ \sqrt{\beta}^{-1}C\sqrt{\alpha}^{-1} & 0 \end{pmatrix}\big)
\] generates a $C_0$-semigroup of contractions by \Cref{rem:easy} (a), (b) and \Cref{prop:genthm}, we infer $\tilde{U}(t)=T(t)\diag(\sqrt{\alpha},\sqrt{\beta})(u_0,v_0)$ for all $t\geq 0$, where $T$ is the semi-group generated by $B$. 

In order to assess whether $T$ is exponentially stable, we apply \Cref{thm:GP}. For this, by  \Cref{rem:easy} (c), the range of $D\coloneqq \sqrt{\beta}^{-1}C\sqrt{\alpha}^{-1}$ is closed (and so is the one of $D^*=\sqrt{\alpha}^{-1}C^*\sqrt{\beta}^{-1}$). By \Cref{rem:easy} (d), $\tilde{U}(0) \in \dom(D)\times \dom(D^*)\cap \ran(D)$. Thus, we need to estimate the resolvent of
\[
  \tilde{B}\coloneqq - \begin{pmatrix} \sqrt{\alpha}^{-1}\gamma\sqrt{\alpha}^{-1} & 0 \\ 0 & 0 \end{pmatrix} - \begin{pmatrix} 0 & -D^* \\ D & 0 \end{pmatrix} \text{ on }H_0 \times \ran(D).
\]
Now, by \Cref{thm:diag}, $z-\tilde{B}$ may be diagonalised in such a way (with invertible transformations with norm bounds uniformly bounded in $z$ as long as $\Re z>-\frac{1}{2}c$) that the transformed operator of $z-\tilde{B}$ reads
\[
\Big(z \begin{pmatrix}1 & 0 & 0 \\ 0 & 1& 0\\ 0 & 0 & 1 \end{pmatrix}+
\begin{pmatrix} {\tilde{\gamma}}^1_z & 0& 0 \\ 0 & 0& 0 \\ 0 & 0 & {\tilde{\gamma}}^2 \end{pmatrix} \\ +\begin{pmatrix}0 & -\tilde{D}^*& 0 \\   \tilde{D} & 0&0 \\ 0 & 0 & 0 \end{pmatrix}\Big).
\]
with $\tilde{D}$ being one-to-one and onto, closed and densely defined and $\Re {\tilde{\gamma}}^1_z \geq c$ and $\Re {\tilde{\gamma}}^2 \geq c$ by \Cref{lem:topleft}. Thus, the claim follows from \Cref{rem:rezc} (a) and \Cref{thm:resB}.
\end{proof}

\begin{remark}\begin{enumerate}
\item[(a)] As a by-product of the proof, one can show that the evolutionary equation
\[
\big(\partial_{t,\nu}\begin{pmatrix} \alpha & 0 \\ 0 & \beta \end{pmatrix}+\begin{pmatrix} \gamma & 0 \\ 0 & 0\end{pmatrix} +\begin{pmatrix} 0 & -C \\ C^* & 0\end{pmatrix} \big)U = F \in L_{2,\nu}(\R;H_0\times \beta^{-1}\ran(C))
\]
is exponentially stable, see the part on hyperbolic type equations in \cite[Section 11.4]{STW22} and \cite{Tro15}, see also \cite{Tro13} for the details. Note that the rationale provided in \cite{STW22} also applies to first order systems even though it has only been formulated for second order type equations.
\item[(b)] Since the resolvent bounds in the proof of \Cref{thm:resB} do not use that $\gamma$ is independent of $z$ as long as it is bounded in $z$ an immediate consequence is that this rationale gives reason to generalisations.\qedhere
\end{enumerate}
\end{remark}
In order to quickly present the result drawn from the latter remark for evolutionary equations, we refer to \cite[Chaper 11]{STW22} for the notational details in general and present the theorem as follows; see, in particular, \cite{Tro15}, see also \cite{Tro13}.
\begin{theorem}\label{thm:expstabS} Let $C\colon \dom(C)\subseteq H_0\to H_1$ be densely defined, closed with closed range. Let $c>0$. Assume $c\leq M_0=M_0^*\in L(H_0)$.  Let $M_1$ be a material law and define $M(z)\coloneqq M_0+z^{-1}M_1(z)$. Assume that there exists $\nu_0>0$ such that $\C_{\Re>-\nu_0}\setminus \dom(M)$ is discrete and for all $z\in \dom(M)$
\[
    \Re zM(z) \geq c \quad (z\in \C_{\Re>-\nu_0}).
\]Then, there exists $\mu>0$ such that for any $(f,g)\in L_{2,\nu}(\R;H_0\times \ran(C))\cap L_{2,-\mu}(\R;H_0\times \ran(C))\eqqcolon L_{2,\nu}\cap L_{2,-\mu}(\R;H_0)$ the unique solution $(u,q) \in L_{2,\nu}(\R;H_0)$ of
\[
    \big( \partial_{t,\nu} \begin{pmatrix} M_0 & 0 \\ 0 & 1\end{pmatrix} + \begin{pmatrix} M_1(\partial_{t,\nu}) & 0 \\ 0 & 0 \end{pmatrix} + \begin{pmatrix} 0& -C^* \\ C & 0 \end{pmatrix}\big)\begin{pmatrix} u \\ q \end{pmatrix} =\begin{pmatrix} f \\ 0 \end{pmatrix}
\]
satisfies $(u,q) \in L_{2,-\mu}(\R;H_0\times \ran(C))$.
\end{theorem}

\section{An application to Maxwell's equations with full damping}\label{sec:fulldamping}

For convenience of the reader we present a small example concerning Maxwell's equations next. 
Let $\Omega\subseteq \R^3$ be open and introduce
\[
   \curl \colon \dom(\curl)\subseteq L_2(\Omega)^3 \to L_2(\Omega)^3,
\]
 the distributional $\curl$-operator with maximal domain in $L_2(\Omega)^3$. Note that this operator is densely defined and closed; we put
\[
   \curl_0\coloneqq \curl^*.
\]Let $c\leq \varepsilon=\varepsilon^*,\mu=\mu^*\in L(L_2(\Omega)^3)$ and $\sigma \in L(L_2(\Omega)^3)$ such that $\Re \sigma\geq c$ for some $c>0$.
\begin{hypothesis}\label{hyp:closedrange}
Assume that $\Omega$ is such that $\ran(\curl)\subseteq L_2(\Omega)^3$ is closed.
\end{hypothesis}
\begin{remark} By recent results in \cite{PW26}, apart from conditions like $\Omega$ to be a bounded weak Lipschitz domain, in which case $\dom(\curl)\cap \ker(\curl)\hookrightarrow L_2(\Omega)^3$ is compact  (see \cite{Pi84}), the closed range condition can be achieved for instance for particular convex domains that are bounded in 2 directions; that is, convex domains being contained in a possibly doubly infinite cylinder with bounded cross-section.
\end{remark}
\begin{theorem}\label{thm:expstabMax} Assume \Cref{hyp:closedrange}. Then there exists $\rho>0$ such that any classical solution of
\[
\partial_t \begin{pmatrix} \varepsilon & 0 \\ 0 & \mu \end{pmatrix}(u,q) =-\begin{pmatrix} \sigma & 0 \\ 0 & 0 \end{pmatrix}(u,q)- \begin{pmatrix} 0 & -\curl \\ \curl_0 & 0 \end{pmatrix}\big)(u,q)\]
with $(u(0),q(0))\in \dom(\curl_0)\times \dom(\curl)\cap \mu^{-1}\ran(\curl)$
decays exponentially\end{theorem}
\begin{proof}
The proof is a direct consequence of \Cref{thm:mt}.
\end{proof}
%

\end{document}